\documentclass[11pt]{article}
\usepackage{amsfonts}
\usepackage{amsmath}
\usepackage{graphicx}
\usepackage{bbm}
\usepackage[normalem]{ulem}
\usepackage[textsize=small]{todonotes}
\usepackage{todonotes}
\usepackage{enumerate}

\numberwithin{equation}{section}

\definecolor{aqua}{rgb}{0.0, 1.0, 1.0}
\definecolor{boo}{rgb}{1.0, 0.0, 1.0}

\newtheorem{theorem}{Theorem}[section]

\newtheorem{definition}[theorem]{Definition}

\newtheorem{lemma}[theorem]{Lemma}

\newtheorem{proposition}[theorem]{Proposition}

\newtheorem{assumption}[theorem]{Assumption}
\newenvironment{proof}[1][Proof]{\textbf{#1.} }{\ \rule{0.5em}{0.5em}}

\newtheorem{preremark}{Remark}[section]
  \newenvironment{remark}%
    {\begin{preremark}\rm}{\end{preremark}}

 \newcommand{\RR}{\mathbb{R}} 
  \newcommand{\cle}{\mathcal{E}} 
   \newcommand{\clg}{\mathcal{G}}  
  \newcommand{\cla}{\mathcal{A}} 
   \newcommand{\clh}{\mathcal{H}} 
  \newcommand{\clc}{\mathcal{C}} 
   \newcommand{\Om}{\Omega}
   \newcommand{\veps}{\varepsilon}
   \newcommand\st{\bgroup\markoverwith{\textcolor{red}{\rule[0.5ex]{2pt}{1.5pt}}}\ULon}
   \definecolor{gold}{RGB}{178,34,34}
   \definecolor{bl}{RGB}{34,178,178}

\begin{document}
\title{\textbf{Large deviation principles for stochastic dynamical systems with a fractional Brownian noise}}
\author{Amarjit Budhiraja and Xiaoming Song}

\date{}
\maketitle
\begin{abstract}
We study small noise large deviation asymptotics for stochastic differential equations with a multiplicative noise given as a fractional Brownian motion $B^H$ with Hurst parameter $H>\frac12$. 
The solutions of the stochastic differential equations are defined pathwise under appropriate conditions on the coefficients.
The ingredients in the proof of the large deviation principle, which include  a  variational representation for nonnegative functionals of fractional Brownian motions and a general sufficient condition for a LDP for a collection of functionals of a fractional Brownian motions, have a
 broader applicability than the model considered here.


\end{abstract}

\section{Introduction}
The goal of this work is to study large deviation asymptotics for singularly perturbed ordinary differential equations with a multiplicative fractional Brownian noise. Specifically, we consider 
multidimensional stochastic differential equation (SDE) of the form
 \begin{equation}\label{eq-1-1}
X_t=x_0+\int_0^tb(s,X_s)ds+\sqrt{\veps}\int_0^t\sigma(s,X_s)dB_s^H,
\end{equation}
or equivalently
\begin{equation*}
X_t^{i}=x_0^i+\int_0^tb^i(s,X_s)ds+\sqrt{\veps}\sum_{j=1}^d\int_0^t\sigma^{i,j}(s,X_s)dB_s^{H,j},\, i=1,\dots,m,
\end{equation*}
where $x_0\in\mathbb{R}^m$, $\veps \in (0,\infty)$ is a small paramater, $b:[0,1]\times \RR^m \to \RR^m$,
$\sigma: [0,1]\times \RR^m \to \RR^{m\times d}$ are suitable coefficients, $B^H$ is a $d$-dimensional fractional Brownian motion with Hurst parameter $H\in(\frac{1}{2},1)$,
 and the integral with respect to $B^H$ is a pathwise Riemann-Stieltjes integral in the sense of \cite{Y}. The existence and uniqueness, and the properties of the solution to the above SDE have been studied by many authors (cf. \cite{MR1302388},  \cite{MR1382288}, \cite{MR1793494},   \cite{MR1906118}, \cite{MR1798553}, \cite{NR},  \cite{MR2397797}).

When $H= 1/2$, $B^H\equiv B$ is a  $d$-dimensional standard Brownian motion and for this case the large deviation properties of 
\eqref{eq-1-1} have been well studied starting from the seminal work of Freidlin and Wentzell \cite{MR722136} (cf. \cite{MR590626},  \cite{MR755154}, \cite{MR758258},  \cite{DE}). Since for $H\neq 1/2$ the fractional Brownian motion $B^H$ is not a semimartingale, the classical stochastic calculus methods are not readily applicable for the study of  stochastic integrals with respect to $B^H$. Nevertheless, fractional Brownian motion has a nice representation as  a Volterra type integral with respect to a standard Brownian motion (cf. \cite{DU}):
\begin{equation}\label{eq:bhbrep}
B_t^H= \int_0^1K_H(t,s)dB_s, \; t \in [0,1],
\end{equation}
where $B$ is a standard Brownian motion and $K_H:[0,1]\times [0,1] \to \RR$ is a suitable kernel (see \eqref{K-H}).
Several authors have studied stochastic Volterra equations of the form
\begin{equation}\label{vol-fbm}
X_t=x_0+\int_0^t K(t,s)b(s,X_s)ds+\int_0^tK(t,s)\sigma(s,X_s)dB_s,
\end{equation}
where the above stochastic integral is the usual It\^{o} integral and $K$ is some kernel. 
Stochastic Volterra equations with regular kernels and driven by a finite-dimensional Brownian motion were first investigated by Berger and Mizel \cite{MR581430} and since then such equations have been widely studied  under more general conditions, including settings where $K$ is the  fractional Brownian motion kernel $K_H$ (cf. \cite{MR781420}, \cite{MR1071815}, \cite{MR1886561}, \cite{MR1881594}, \cite{MR2422961}).

 Large deviations for small noise stochastic Volterra equations have also been studied in many works including,  \cite{MR1748725}, \cite{MR2020569}, \cite{MR2413840} and \cite{MR1978984}.  In particular, using Azencott's method (cf. \cite{MR590626}), Nualart and Rovira \cite{MR1748725} proved a large deviation principle (LDP) in the space of continuous functions and later Lakhel \cite{MR2020569} improved this result by establishing a LDP in a Besov-Orlicz space. The last twenty years have seen the emergence of a  method for the study of large deviation properties of stochastic dynamical systems driven by a Brownian motion (BM) that is based on techniques from the theory of weak convergence and stochastic control.  This method, which was originally introduced for the setting of finite dimensional SDE driven by BM in Bou\'{e} and  Dupuis \cite{MR1675051} and then later extended to stochastic dynamical systems with an infinite dimensional BM in Budhiraja and Dupuis \cite{BD} has now been used for a broad range of problems (see \cite{buddupbook} and references therein). 
 The starting point in this approach   is a variational formula for exponential functionals of a BM using which the problem of large deviations is reduced to proving certain weak convergence properties of controlled analogues of the underlying stochastic dynamical systems describing the state processes.  This approach was used in Zhang \cite{MR2413840} for studying LDP for small noise 
 stochastic Volterra equations with very general types of kernels  including the  fractional Brownian motion kernel $K_H$.

The SDE  \eqref{eq-1-1} considered in this work cannot in general be written as a stochastic Volterra equation of the form 
\eqref{vol-fbm}. Writing the kernel $K_H$ as $K_H(t,s) = k_H(t,s) \mathbf{1}_{[0,t]}(s)$ where $k_H$ is as in \eqref{eq:khts}
and using the representation \eqref{eq:bhbrep}, one can formally write the pathwise integral $\int_0^t\sigma(s,X_s)dB^H_s$ in \eqref{eq-1-1}
as
\begin{eqnarray*}
\int_0^t\sigma(s,X_s)dB^H_s&=&\int_0^t\sigma(s,X_s)k_H(t,s)dB_s+\int_0^t\int_s^t[\sigma(u,X_u)-\sigma(s,X_s)]\partial_uk_H(u,s)du\, \delta B_s\\
&&+\int_0^t\int_0^uD^B_s\sigma(u,X_u)\partial_u k_H(u,s)ds\,du\\
&=&\int_0^t\int_s^t\sigma(u,X_u)\partial_uk_H(u,s)\delta B_s+\int_0^t\int_0^uD^B_s\sigma(u,X_u)\partial_u k_H(u,s)ds\,du,
\end{eqnarray*}
where $dB$ and $\delta B$ denote the It\^{o} integral  and  the Skorohod integral respectively with respect to the Brownian motion $B$, and $D^B$ is the Malliavin derivative with respect to $B$. 
In particular, even though $B^H$ has a simple representation in terms of a standard Brownian motion $B$ and the fBm kernel $K_H$ (as given in \eqref{eq:bhbrep}) the stochastic integral on the left side of the above display does not admit a similar tractable representation. Thus, one cannot directly use the methods of \cite{MR1748725,MR2020569,MR2413840,MR1978984} to obtain a large deviation principle for the solution of \eqref{eq-1-1}.  It is therefore of interest to develop a general approach for studying large deviation results for functionals of a fractional Brownian motion, analogous to the one developed in
\cite{MR1675051, BD} for Brownian motion systems, that does not  make an explicit use of the representation  \eqref{eq:bhbrep}.

The starting point of our work is a nice extension of the variational representation for functionals of Brownian motions given in \cite{MR1675051, BD} to the general setting of an abstract Wiener space that was established in \cite{Zhang}. By explicitly identifying the abstract Wiener space and related constructions for a fractional Brownian motion (e.g. monotonic resolution 
 of the identity operator on the Cameron-Martin space,  a filtration on the associated abstract Wiener space, the Skorohod integral) we apply \cite[Theorem 3.2]{Zhang} to give  a convenient representation for functionals of a fractional Brownian motion that is analogous to Brownian motion representations in \cite{MR1675051, BD} (see  Proposition \ref{representation}).
 Next, by adapting ideas from \cite{MR1675051, BD}, this representation is used to give a general sufficient condition for a LDP to hold for functionals of a fractional Brownian motion. This condition is analogous to sufficient conditions given for Brownioan noise driven systems in \cite{BD, MR2435853} and for systems with a Poisson noise in \cite{BuDuMA2} which have found broad applicability (see \cite{buddupbook} and references therein). With this sufficient condition, the proof of a LDP for the solution of \eqref{eq-1-1},
 becomes very similar to proofs for SDE driven by BM. Once again the key step is to verify certain weak convergence properties of controlled analogues of the original stochastic differential equation. The control processes this time take a different form and  are given as suitable random non-anticipative elements 
 of the Cameron-Martin space associated with the fractional Brownian motion (see the definition of classes $\cla$ and $\cla_b$
 in Section \ref{sec:varrep}). The verification of these weak convergence properties uses various results for 
 the SDE  \eqref{eq-1-1} that have been developed in \cite{NR} and a key estimate on the H\"{o}lder norms of the solutions in terms of those of the driving paths (see Proposition \ref{existence-uniqueness}).

The paper is organized  as follows. Section 2 introduces some notation and background results. Section 3 gives our main assumptions 
and states the main results.  Proofs of these results are provided in Section 4.

\section{Notation and Preliminaries }
In this section we recall the definition of a fractional Brownian motion and introduce 
some related terminology  and notation.
\subsection{Fractional integrals and derivatives}
Let $a,b\in \mathbb{R}$ with $a<b$. Denote by $L^p([a,b]),\,p\geq1$, the usual space of Lebesgue 
measurable functions $f:\,[a,b]\rightarrow\mathbb{R}$ for which $\Vert f\Vert_{L^p}<\infty$, where
\begin{equation*}
\Vert f\Vert_{L^p}=
\begin{cases}
\left(\int_a^b\vert f(t)\vert^p dt\right)^{1/p},\,\text{if $1\leq p<\infty$}\\
\text{ess} \sup\{\vert f(t)\vert:\,t\in[a,b]\},\,\text{if $p=\infty$}. 
\end{cases}
\end{equation*}
Let $f\in L^{1}\left( [a,b]\right) $ and $\alpha >0.$ The left-sided and right-sided fractional Riemann-Liouville
integrals of $f$ of order $\alpha $ are defined for almost all $t\in \left(
a,b\right) $ by
\[
I_{a+}^{\alpha }f\left( t\right) =\frac{1}{\Gamma \left( \alpha \right) }%
\int_{a}^{t}\left( t-s\right) ^{\alpha -1}f\left( s\right) ds
\]%
and
\[
I_{b-}^{\alpha }f\left( t\right) =\frac{\left( -1\right) ^{-\alpha }}{\Gamma
\left( \alpha \right) }\int_{t}^{b}\left( s-t\right) ^{\alpha -1}f\left(
s\right) ds,
\]%
respectively, where $\left( -1\right) ^{-\alpha }=e^{-i\pi \alpha }$ and $%
\Gamma \left( \alpha \right) =\int_{0}^{\infty }r^{\alpha -1}e^{-r}dr$ is
the Euler gamma function. Let $I_{a+}^{\alpha }(L^{p}([a,b]))$ (resp. $%
I_{b-}^{\alpha }(L^{p}([a,b]))$) be the image of $L^{p}([a,b])$ under the operator $%
I_{a+}^{\alpha }$ (resp. $I_{b-}^{\alpha }$). If $f\in I_{a+}^{\alpha
}\left( L^{p}([a,b])\right) \ $ (resp. $f\in I_{b-}^{\alpha }\left( L^{p}([a,b])\right) $)
and $0<\alpha <1$ then the Weyl derivatives are defined as
\begin{equation} \label{1.1}
D_{a+}^{\alpha }f\left( t\right) =\frac{1}{\Gamma \left( 1-\alpha \right) }%
\left( \frac{f\left( t\right) }{\left( t-a\right) ^{\alpha }}+\alpha
\int_{a}^{t}\frac{f\left( t\right) -f\left( s\right) }{\left( t-s\right)
^{\alpha +1}}ds\right) 1_{(a,b)}(t) 
\end{equation}%
and
\begin{equation} \label{1.2}
D_{b-}^{\alpha }f\left( t\right) =\frac{\left( -1\right) ^{\alpha }}{\Gamma
\left( 1-\alpha \right) }\left( \frac{f\left( t\right) }{\left( b-t\right)
^{\alpha }}+\alpha \int_{t}^{b}\frac{f\left( t\right) -f\left( s\right) }{%
\left( s-t\right) ^{\alpha +1}}ds\right) 1_{(a,b)}(t)   
\end{equation}%
for almost all $t\in(a,b)$ (the convergence of the integrals at the singularity $%
s=t$ holds pointwise for almost all $t\in \left( a,b\right) $ if $p=1$ and
 in $L^{p}$ if $1<p<\infty $).

\bigskip
Denote by $C([a,b]:\mathbb{R}^d)$ the space of $\mathbb{R}^d$-valued  continuous functions on the interval $[a,b]$. Note that $C([a,b]:\mathbb{R}^d)$ is a separable Banach space equipped with the following norm
\[
\Vert f\Vert_{a,b,\infty }=\sup_{a\leq r\leq b}|f(r)|.
\]%
For any $\lambda \in (0,1)$, we denote by $C^{\lambda }([a,b]:\mathbb{R}^d)$ the space of $\mathbb{R}^d$-valued $%
\lambda $-H\"{o}lder continuous functions on the interval $[a,b]$. We will
make use of the notation%
\[
\left\| f\right\| _{a,b,\lambda }=\sup_{a\leq \theta <r\leq b}\frac{%
|f(r)-f(\theta )|}{|r-\theta |^{\lambda }},
\]%
if $f\in C^{\lambda }([a,b]:\mathbb{R}^d)$.
In the case of $d=1$, we abbreviate $C^\lambda([a,b]:\mathbb{R})$ as  $C^{\lambda }([a,b])$.

Suppose that $f\in C^{\lambda }([a,b])$ and $g\in C^{\mu }([a,b])$ with $\lambda
+\mu >1$. Then, from the classical paper by Young \cite{Y}, the
Riemann-Stieltjes integral $\int_{a}^{b}fdg$ exists. The following
proposition can be regarded as a fractional integration by parts formula,
and provides an explicit expression for the integral $\int_{a}^{b}fdg$ in
terms of fractional derivatives (cf. \cite{Z}).

\begin{proposition}
\label{p1} Suppose that $f\in C^{\lambda }([a,b])$ and $g\in C^{\mu }([a,b])$
with $\lambda +\mu >1$. Let $0<\alpha<1$, ${\lambda }>\alpha $ and $\mu >1-\alpha $. Then
the Riemann-Stieltjes integral $\int_{a}^{b}fdg$ exists and it can be
expressed as%
\begin{equation}
\int_{a}^{b}fdg=(-1)^{\alpha }\int_{a}^{b}D_{a+}^{\alpha }f\left( t\right)
D_{b-}^{1-\alpha }g_{b-}\left( t\right) dt,  \label{1.8}
\end{equation}%
where $g_{b-}\left( t\right) =g\left( t\right) -g\left( b\right) $.
\end{proposition}

When we consider the interval $[a,b]=[0,1]$, we will use the following simplified notations
\[
\Vert f\Vert_{\infty }=\Vert f\Vert_{0,1,\infty }, \;\; \Vert f\Vert_{\lambda}=\Vert f\Vert_{0,1,\lambda}
\]
for  $f$ in $C([0,1]:\RR^d)$ and 
 $C^\lambda([0,1]:\RR^d)$, respectively.

For $0<\alpha<1$, let $W_0^{\alpha,\infty}([0,1]:\mathbb{R}^d)$ be the space of $\mathbb{R}^d$-valued measurable functions $f:\,[0,1]\rightarrow \mathbb{R}^d$ such that 
\[
\Vert f\Vert_{\alpha,\infty}=\sup_{0\leq t\leq 1} \left\{|f(t)|+\int_0^t\frac{|f(t)-f(s)|}{(t-s)^{\alpha+1}}ds\right\}<\infty.
\]
It is easy to see that 
\[
 C^{\alpha+\varepsilon}([0,1]:\RR^d)\subseteq W_0^{\alpha, \infty}([0,1]:\mathbb{R}^d)
\]
for any $\varepsilon>0$.  On the other hand, by the Garsia-Rademich-Rumsey inequality (see \cite{MR267632}), it follows that
\[
 W_0^{\alpha, \infty}([0,1]:\mathbb{R}^d)\subseteq C^{\alpha-\varepsilon}([0,1]:\RR^d)
 \]
 for any $0<\varepsilon<\alpha$.\\

\subsection{Fractional Brownian motion}

We now recall the definition of a fractional Brownian motion.

 \begin{definition}
	 \label{def:fbm}
For  $H\in(0,1)$, a $d$-dimensional fractional Brownian motion (fBm for short) $B^H=\{B^H_t:\,t\in[0,1]\}$ with Hurst parameter $H$ defined on a complete probability space $(\Omega, \mathcal{F}, \mathbb{P})$ is a centered Gaussian process whose covariance matrix $R_H=(R_H^{i,j})_{1\leq i,\,j\leq d}$ is given by 
\begin{equation}\label{covariance}
R_H^{i,j}(s, t)=\mathbb{E}(B^{H,i}_sB^{H,j}_t)=\frac{1}{2}(s^{2H}+t^{2H}-|t-s|^{2H})\delta_{i,j},
\; s,\,t\in[0,1],
\end{equation}
 where $\delta$ is the Kronecker delta function.
\end{definition}

For $H=\frac{1}{2}$, the process $B^{\frac{1}{2}}$ is a $d$-dimensional standard Brownian motion. For $H\neq\frac{1}{2}$, the fBm $B^H$ is not a semimartingale. 
It follows from (\ref{covariance}) that
\begin{equation*}
\mathbb{E}(\vert B_t^H-B_s^H\vert^2)=d\vert t-s\vert^{2H},\; t,s \in [0,1].
\end{equation*}
From the above property along with Kolmogorov's continuity criterion it follows that the sample paths of $B^H$ are a.s. H\"{o}lder continuous of order $\beta$ for all $\beta<H$.

In the sequel we consider the canonical probability space $(\Omega,\mathcal{F},\mathbb{P})$, where $\Omega=C_0([0,1]:\mathbb{R}^d)$ is the space of continuous functions null at time $0$, $\mathcal{F}=\mathcal{B}(C_0([0,1]:\,\mathbb{R}^d))$ is the Borel $\sigma$-algebra and $\mathbb{P}$ is the unique $d$-dimensional probability measure such that the  canonical process $B^H=\{B^H_t(\omega)=\omega(t):\,t\in[0,1]\}$ is a $d$-dimensional fractional Brownian motion with Hurst parameter $H$. Consider the canonical filtration given by $\{\mathcal{F}^H_t:\,t\in[0,1]\}$, where $\mathcal{F}^H_t =\sigma\{B^H_s:\,0\leq s\leq t\}\vee \mathcal{N}$ and $\mathcal{N}$ is the set of the $\mathbb{P}$-negligible events.

 Let $F(a,b,c;z)$ denote the Gauss hypergeometric function defined for any $a,\,b,\,c,\, z\in\mathbb{C}$ with $|z|<1$ and $c\neq 0,-1,-2,\dots$ by
\[
F(a,b,c;z)=\sum_{k=0}^{\infty}\frac{(a)_k(b)_k}{(c)_k}z^k,
\]
where $(a)_0=1$ and $(a)_k=a(a+1)\dots(a+k-1)$ is the Pochhammer symbol.

Let $B=\{B_t=(B_t^1,\dots,B_t^d),\,t\in[0,1]\}$ be a standard $d$-dimensional Brownian motion.
Then from \cite{DU} we have that the process
\begin{equation}\label{bm-2-fbm}
B_t^H= \int_0^1K_H(t,s)dB_s, \; t \in [0,1]
\end{equation}
defines a fBm with Hurst parameter $H$, where 
  \begin{equation}\label{K-H}
K_H(t,s)= k_H(t,s) \mathbf{1}_{[0,t]}(s),
\end{equation}
for $0\le s\le t$
\begin{equation}\label{eq:khts}
k_H(t,s) = \frac{c_H}{\Gamma\left(H+\frac{1}{2}\right)}(t-s)^{H-\frac{1}{2}}F\left(H-\frac{1}{2},\frac{1}{2}-H,H+\frac{1}{2};1-\frac{t}{s}\right),
\end{equation}
$c_H=\left[\frac{2H\Gamma\left(\frac{3}{2}-H\right)\Gamma\left(H+\frac{1}{2}\right)}{\Gamma(2-2H)}\right]^{1/2}$,
 and $\Gamma(\cdot)$ as before is the gamma function. \newline

The next lemma (cf. Theorem 2.1 in \cite{DU} and (10.22) in \cite{SKM}) will play an important role in the construction of the Cameron-Martin space in the abstract Wiener space associated with a fBm.

\begin{lemma}\label{int-trans}
For $H\in(0,1)$, consider the integral transform
\begin{eqnarray}\label{K-H-mapping}
(K_Hf)(t)&=&\int_0^1K_H(t,s)f(s)ds\nonumber\\
&=&\frac{c_H}{\Gamma(H+\frac{1}{2})}\int_0^1 (t-s)^{H-\frac{1}{2}}F\left(\frac{1}{2}-H, H-\frac{1}{2}, H+\frac{1}{2};1-\frac{t}{s}\right) f(s)ds.\nonumber\\
\end{eqnarray}
Then $K_H$ is an isomorphism from $L^2([0,1]:\RR^d)$ onto $I_{0+}^{H+\frac{1}{2}}(L^2([0,1]:\RR^d))$ and
\begin{eqnarray}
K_Hf(t) &=&c_H I^1_{0+}\left(\psi \left(I^{H-\frac{1}{2}}_{0+}(\psi^{-1} f)\right)\right)(t),\, \hbox{if}\ H\geq \frac{1}{2},\; t \in [0,1]\label{H-big}\\
K_Hf(t) &=&c_H I^{2H}_{0+}\left(\psi^{-1} \left(I^{\frac{1}{2}-H}_{0+}(\psi f)\right)\right)(t),\, \hbox{if}\ H\leq \frac{1}{2}, \; t \in [0,1].
\end{eqnarray}
where $\psi(u) = u^{H-\frac{1}{2}}$ for $u \in [0,1]$.
\end{lemma}
\begin{remark}\label{rem-m-1}
For any $a<b$, if $\alpha >\frac{1}{p}$,  then from Theorem 3.6 in \cite{SKM} we have
\[
I_{a+}^{\alpha }\left( L^{p}([a,b])\right) \,\cup \,I_{b-}^{\alpha }\left(
L^{p}([a,b])\right) \subset C^{\alpha -\frac{1}{p}}\left( [a,b]\right) .
\]
In particular, we have $I_{0+}^{H+\frac{1}{2}}(L^2([0,1]:\RR^d))\subset C^H([0,1]:\RR^d)$. Note that $(K_Hf)(0)=0$ for all $f\in L^2([0,1]:\RR^d)$.
Thus $K_H$ maps $L^2([0,1]:\RR^d)$ into $C_0([0,1]:  \mathbb{R}^d)$ and for $f_1, f_2 \in L^2([0,1]:\RR^d)$, $f_1= f_2$ a.e.
if and only if $K_Hf_1 = K_Hf_2$.
\end{remark}
\bigskip

Define $\mathcal{H}_H=\{(K_H\dot{h}^1,\dots,K_H\dot{h}^d):\,\dot{h}=(\dot{h}^1,\dots,\dot{h}^d)\in L^2([0,1]:\RR^d)\}$, that is, any $h\in\mathcal{H}_H$ can be represented as
\[
h(t)=(K_H\dot{h})(t)=\int_0^1K_H(t,s)\dot{h}(s)ds,
\]
for some $\dot{h}\in L^2([0,1]:\RR^d)$. Define a scalar inner product on $\mathcal{H}_H$ by 
\[
\langle h, g\rangle_{\mathcal{H}_H}=\langle K_H\dot{h}, K_H\dot{g}\rangle_{\mathcal{H}_H}=\langle \dot{h}, \dot{g}\rangle_{L^2}.
\]

Then $\mathcal{H}_H$ is a separable Hilbert space with the inner product $\langle \cdot,\cdot\rangle_{\mathcal{H}_H}$. From Remark \ref{rem-m-1}, we  see that $\mathcal{H}_H$ is a subset of $ \Omega=C_0([0,1]:  \mathbb{R}^d)$.

\section{Results}
We now present the main results of this work. The proofs will be provided in Section \ref{sec-4}.
\subsection{A variational representation for functionals of fBm}
\label{sec:varrep}
For  $0<N<\infty$, let $S_N=\{v\in\mathcal{H}_H: \,\frac{1}{2}\Vert v\Vert^2_{\mathcal{H}_H}\leq N\}$. 
Equipped with the weak topology on $\mathcal{H}_H$, $S_N$ can be metrized as a compact Polish space.
Let $\mathcal{A}$ denote  the class of all  $\mathcal{H}_H$-valued random variables $v$ in $L^2((\Omega,\mathcal{F},\mathbb{P}):\mathcal{H}_H)$ such that $v(t)$ is
$\{\mathcal{F}^H_t\}$-measurable for every $t \in [0,1]$. The set of all a.s. bounded elements in $\mathcal{A}$ is denoted by $\mathcal{A}_b$, that is,
\[
\mathcal{A}_b=\{v\in\mathcal{A}: \,\Vert v(\omega)\Vert_{\mathcal{H}_H}\leq N\,\, \mathbb{P}\text{-a.s. for some}\,\, 0<N<\infty\}.
\]

The following variational representation is obtained by making use of \cite[Theorem 3.2]{Zhang}.
\begin{proposition}\label{representation}
Let $f$ be a bounded Borel measurable function on $\Omega$. Then we have
\begin{eqnarray*}
-\log\mathbb{E}(e^{-f(\omega)})&=&\inf_{v\in\mathcal{A}}\mathbb{E}\left(f(\omega+v)+\frac{1}{2}\Vert v\Vert^2_{\mathcal{H}_H}\right)\\
&=&\inf_{v\in\mathcal{A}_b}\mathbb{E}\left(f(\omega+v)+\frac{1}{2}\Vert v\Vert^2_{\mathcal{H}_H}\right).
\end{eqnarray*}
\end{proposition}

\subsection{A general large deviation principle}
We begin by recalling the definition of a large deviation principle.
Let $\mathcal{E}$ be a Polish space (a complete separable metric space)
and let $\{X^\varepsilon:\, \varepsilon\in (0,1)\}$ be a collection of $\cle$-valued random variables.  \begin{definition}
	\begin{enumerate}[(a)]
\item A function $I:\,\mathcal{E}\to [0,\infty]$ is called a rate function on $\mathcal{E}$, if for each $M<\infty$ the level set $\{x\in\mathcal{E}:\,I(x)\leq M\}$ is a compact subset of $\mathcal{E}$. For $A\in\mathcal{B}(\mathcal{E})$, we define $I(A)=\inf_{x\in A}I(x)$.

\item Let $I$ be a rate function on $\mathcal{E}$. A collection $\{X^\varepsilon:\, \varepsilon\in (0,1)\}$ of $\cle$-valued random variables  is said to satisfy a large deviation principle in $\mathcal{E}$, as $\veps \to 0$, with rate function $I$ if the following two conditions hold:
\begin{enumerate}[(i)]
\item Large deviation upper bound. For each open set $G$ in $\mathcal{E}$,
\[
\limsup_{\varepsilon\to0} -\varepsilon\log{\mathbb{P}(X^\varepsilon\in G)}\leq I(G).
\]
\item Large deviation lower bound. For each closed set $F$ in $\mathcal{E}$,
\[
\liminf_{\varepsilon\to0}- \varepsilon\log{\mathbb{P}(X^\varepsilon\in F)}\geq  I(F).
\]
\end{enumerate}

\end{enumerate}
\end{definition}

For  $0<\varepsilon<1$ let $\mathcal{G}^\varepsilon:\,\Omega= C_0([0,1]:\RR^d)\to \mathcal{E}$ be a measurable map. We will now give a sufficient condition  on the maps $\mathcal{G}^\varepsilon$ for a LDP to hold for the collection
of $\cle$-valued random variables
 \begin{equation}\label{eq-x}
 X^\varepsilon(\omega)=\mathcal{G}^\varepsilon(\sqrt{\varepsilon} \omega), \; \omega \in \Om.
 \end{equation}

 Recall that $S_N$ is a compact metric space that is equipped with the weak topology inherited from $\mathcal{H}_H$.
 Also recall the canonical filtration $\{\mathcal{F}^H_t\}$ on $(\Omega,\mathcal{F},\mathbb{P})$ and the  classes $\cla, \cla_b$.
\begin{assumption}\label{assum-G}
There exists a measurable map $\mathcal{G}^0:\, \mathcal{H}_H\to\mathcal{E}$ such that the following conditions hold.
\begin{itemize}
\item[(i)] Consider $0<N<\infty$ and a family of $S_N$-valued random elements  $\{v^\varepsilon\}\subset \mathcal{A}_b$  on $(\Omega,\mathcal{F},\mathbb{P})$  such that $v^\varepsilon$ converges in distribution  to $v$. Then $\mathcal{G}^\varepsilon(\sqrt{\varepsilon}\omega+v^\varepsilon)$ converges to $\mathcal{G}^0(v)$ in distribution.
\item[(ii)] For every $0<N<\infty$, the set 
\[
\Gamma_N=\{\mathcal{G}^0(v):\,v\in S_N\}
\]
is a compact subset of $\mathcal{E}$.
\end{itemize}
\end{assumption}

For  $x\in\mathcal{E}$, define
\begin{equation}\label{eq-rate}
I(x)=\inf_{\{v\in\mathcal{H}_H:\,x=\mathcal{G}^0(v)\}}\left\{\frac{1}{2}\Vert v\Vert_{\mathcal{H}_H}^2\right\}
\end{equation}
whenever $\{v\in\mathcal{H}_H: x=\mathcal{G}^0(v)\}\neq\emptyset$, and $I(x)=\infty$ otherwise.\newline

The following theorem gives a LDP under Assumption \ref{assum-G}.

\begin{theorem} \label{general-LP}
Let $X^\varepsilon$ be as defined in \eqref{eq-x}. Suppose that $\{\mathcal{G}^\varepsilon\}$ satisfies Assumption \ref{assum-G}. Then the family $\{X^\varepsilon:\varepsilon\in(0,1)\}$ satisfies a large deviation principle in $\mathcal{E}$, as $\veps \to 0$, with rate function $I$ defined in \eqref{eq-rate}.
\end{theorem}

\subsection{SDE with a multiplicative fractional noise}
 For  $\varepsilon>0$, consider the following stochastic differential equation
 \begin{equation}\label{small-noise SDE}
 X_t^\varepsilon=x_0+\int_0^tb(s,X_s^\varepsilon)ds+\sqrt{\varepsilon}\int_0^t\sigma(s,X_s^\varepsilon)dB^H_s,\, t\in[0,1].
\end{equation}
where $B^H$ is a $d$-dimensional fBm with $H>1/2$.
For wellposedness of the above equation we introduce the following condition on the coefficients $b$ and $\sigma$.
For a matrix $A=(a^{i,j})_{m\times d}$ and a vector $y=(y^i)_{m\times 1}$, denote $|A|^2=\sum_{i,j}|a^{i,j}|^2$ and $|y|^2=\sum_{i}|y^i|^2$.
\begin{assumption}\label{assum-b-sigma}
\begin{itemize}
\item[(i)] The function $\sigma:\,[0,1]\times \mathbb{R}^m\to \mathbb{R}^{m\times d}$ is differentiable in the space variable $x\in \RR^m$, and there exist  constants $M>0$, $1-H<\lambda\leq1$, $\frac{1}{H}-1<\gamma\leq 1$, and for every $N>0$ there exists $M_N>0$, such that the following hold:
\[
|\sigma(t,x)-\sigma(t,y)|\leq M|x-y|,\, \mbox{ for all } (t,x) \in  [0,1]\times \mathbb{R}^m,
\]
\[
|\partial_{x_i}\sigma(t,x)-\partial_{y_i}\sigma(t,y)|\leq M_N|x-y|^{\gamma}, \mbox{ for all } t\in[0,1], x, y \in \mathbb{R}^m  \mbox{ such that }|x|\vee |y|\leq N,
\]
\[
|\sigma(t,x)-\sigma(s,x)|+|\partial_{x_i}\sigma(t,x)-\partial_{x_i}\sigma(s,x)|\leq M|t-s|^\lambda, \mbox{ for all } x\in\mathbb{R}^m, \mbox{ and } t,\,s\in[0,1],
\]
for each $i=1,\dots,m$.
\item[(ii)] There exists $L>0$ such that the following properties are satisfied:
\[
|b(t,x)-b(t,y)|\leq L|x-y|,\; |b(t,x)|\leq L(1+|x|)\,\mbox{ for all } x,\,y\in\mathbb{R}^m \,\mbox{ and } t\in[0,1].
\]

\end{itemize}
\end{assumption}
\begin{remark}
From Assumption \ref{assum-b-sigma}(i) it follows that there exists a  $K>0$ such that
\begin{equation}\label{linear-sig}
|\sigma(t,x)|\leq K(1+|x|),\,\mbox{ for all } (t,x) \in  [0,1]\times \mathbb{R}^m.
\end{equation}
\end{remark}
\bigskip

Let $g: [0,1] \to \RR^d$ be such that for any $0<\varepsilon<H$, $g \in C^{H-\varepsilon}([0,1]:\mathbb{R}^d)$.  
Consider the deterministic equation on $\mathbb{R}^m$
\begin{equation}\label{deterministic}
x_t=x_0+\int_0^tb(s,x_s)ds+\int_0^t\sigma(s,x_s)dg_s,
\end{equation}
 where $x_0\in\mathbb{R}^m$
 and the last integral is interpreted as a Riemann-Stieltjes integral.

The  wellposedness in  part (a) of the next proposition is taken from \cite[Theorem 5.1]{NR}. The estimate given 
in part (b) of the proposition will play an important role in the analysis and will be established in Section \ref{sec:pfpartb}.

\begin{proposition}\label{existence-uniqueness}
Suppose that the coefficients $b$ and $\sigma$ satisfy Assumption \ref{assum-b-sigma}. Then
the following hold.
\begin{itemize}
\item[(a)]  Equation 
\eqref{deterministic} has a unique solution $x\in W_0^{\alpha,\infty}([0,1]:\mathbb{R}^m)$  for any $\alpha \in (1-H , \min\{\frac{1}{2}, \lambda,\frac{\gamma}{1+\gamma} \})$. Moreover, the solution is $(1-\alpha)$-H\"{o}lder continuous.
\item[(b)] For any $\alpha \in (1-H , \min\{\frac{1}{2}, \lambda,\frac{\gamma}{1+\gamma} \})$ and any $0<\delta<\alpha-(1-H)$,  the following estimates hold:
\begin{equation}\label{equ-4-2}
\Vert x\Vert_{\infty}\leq C_1(1+|x_0|)\exp\{C_2\Vert g\Vert_{1-\alpha+\delta}^\kappa\},
\end{equation}
\begin{eqnarray}\label{equ-m-4-2}
\Vert x\Vert_{1-\alpha}\leq C_3(1+|x_0|)(1+\Vert g\Vert_{1-\alpha+\delta}^\kappa)(1+\Vert g\Vert_{1-\alpha+\delta})(1+\exp\{C_2\Vert g\Vert_{1-\alpha+\delta}^\kappa\}),
\end{eqnarray}
where  $\kappa=\frac{1}{1-\alpha}$ and the constants $C_1$, $C_2$ and $C_3$ depend on $\alpha,\,\delta$ and all the constants appearing in Assumption \ref{assum-b-sigma}.
\end{itemize} 
\end{proposition}
\begin{remark}\label{rem:exisuniq}

As observed above Proposition \ref{p1}, for $f\in C^{\mu_1 }([a,b]:\RR)$ and $h\in C^{\mu_2 }([a,b]:\RR)$ with $\mu_1
+\mu_2 >1$,  the
Riemann-Stieltjes integral $\int_{a}^{b}f_sdh_s$ exists. 
From Assumption  \ref{assum-b-sigma}  and the observation that
 $W_0^{\alpha, \infty}([0,1]:\mathbb{R}^m)\subseteq C^{\alpha-\varepsilon}([0,1]:\RR^m)
 $
 for any $0<\varepsilon<\alpha$, it follows that if for some $\alpha>1-H$, $x \in W_0^{\alpha, \infty}([0,1]:\mathbb{R}^m)$ then
 $s\mapsto \sigma(s, x(s)) \in C^{\delta}([0,1]:\RR^m)$ for some $\delta> 1-H$. Thus, from  properties of $g$ it 
follows that  the integral
$\int_0^t\sigma(s,x_s)dg_s$ is well defined as a Riemann-Stieltjes integral.

\end{remark}

  In Lemma \ref{Holder} it will be seen that,  $\mathcal{H}_H \subset  C^H([0,1]:\RR^d)$ and thus from 
  Proposition \ref{existence-uniqueness}, under Assumption  \ref{assum-b-sigma},  equation \eqref{deterministic} with $g$ replaced by $v$ has a unique solution.
For $f \in C([0,1]:\mathbb{R}^m)$, define
$$
\clc_f \doteq \left\{v\in\mathcal{H}_H:\, f_t=x_0+\int_0^tb(s,f_s)ds+\int_0^t\sigma(s,f_s)dv_s \mbox{ for all } t \in [0,1]\right\}$$
and let 
\begin{equation}\label{eq:rfmul}
I(f)=\inf_{v\in \clc_f}\frac{1}{2}\Vert v\Vert_{\mathcal{H}_H}^2
\end{equation}
whenever $\clc_f\neq \emptyset$, and $I(f)=\infty$ otherwise.\\

We now return to SDE \eqref{small-noise SDE}. Recall that the sample paths of $B^H$ are a.s. in $C^{H-\varepsilon}([0,1]:\mathbb{R}^d)$ for 
 any $0<\varepsilon<H$. Thus as an immediate consequence of  Proposition \ref{existence-uniqueness} and  Fernique's theorem (cf. \cite{MR0413238}) we have the following result. 

\begin{proposition}\label{prop-sde}
Let Assumption \ref{assum-b-sigma} be satisfied.
 \begin{itemize}
\item[(a)] For any $\alpha \in (1-H , \min\{\frac{1}{2}, \lambda,\frac{\gamma}{1+\gamma} \})$,  the SDE 
\eqref{small-noise SDE} has a unique pathwise solution $X^\varepsilon$ with $X^\varepsilon\in W_0^{\alpha,\infty}([0,1]:\mathbb{R}^m)$ a.s. Moreover, for $\mathbb{P}$-almost all $\omega\in \Omega$, the solution $X^\varepsilon(\omega)$ is $(1-\alpha)$-H\"{o}lder continuous. 
\item[(b)] The solution $X^\varepsilon$ to the SDE \eqref{small-noise SDE} satisfies
\begin{equation}\label{equ-4-2-1}
\sup\limits_{\varepsilon\in(0,1)}\mathbb{E}\Vert X^\varepsilon\Vert_{1-\alpha}^p<\infty, \mbox{ for all } p\geq 1.
\end{equation}

\end{itemize} 
\end{proposition}

The following is the main result of this work.
\begin{theorem} \label{LPmult}
Suppose that Assumption \ref{assum-b-sigma} is satisfied and let for $\veps \in (0,1)$, $X^\varepsilon$ be defined as in \eqref{small-noise SDE}.  Then the family $\{X^\varepsilon: \veps \in (0,1)\}$ satisfies a large deviation principle in $C([0,1]:\mathbb{R}^m)$
with rate function $I$ defined in \eqref{eq:rfmul}.
\end{theorem}

\section{Proofs}\label{sec-4}
In this section we give the proofs of Proposition \ref{representation}, Theorem \ref{general-LP}, Proposition \ref{existence-uniqueness}(b) and
Theorem \ref{LPmult}.
\subsection{Proof of Proposition \ref{representation}}
\label{sec:4.1}
In this subsection we describe the abstract Wiener space for $B^H$ and, using \cite[Theorem 3.2]{Zhang}, provide the proof of the variational representation for functionals of fractional Brownian motion given in Proposition \ref{representation}.

The following elementary lemma shows that functions in $\mathcal{H}_H$ are  $H$-H\"{o}lder continuous.
\begin{lemma}\label{Holder}
Let $h\in\mathcal{H}_H$. Then $h\in C^H([0,1]:\RR^d)$, and $\Vert h\Vert_\infty\leq  \Vert h\Vert_{\mathcal{H}_H}$ and $\Vert h\Vert_{H}\leq  \Vert h\Vert_{\mathcal{H}_H}$. 
\end{lemma}
\begin{proof}
Without loss of generality, we assume $d=1$.  
For  $h=K_H\dot{h}\in\mathcal{H}_H$ and  $t\in[0,1]$, we have, from the integral representation \eqref{bm-2-fbm},
\begin{equation*}
h(t)=\int_0^1K_H(t,u)\dot{h}(u)du=\mathbb{E}\left(\int_0^1K_H(t,u)dB_u\int_0^1\dot{h}(u)dB_u\right)=\mathbb{E}\left(B_t^H\int_0^1\dot{h}(u)dB_u\right),
\end{equation*}
where $B$ is a standard Brownian motion
and consequently $B^H_t = \int_0^1K_H(t,u)dB_u$  is a fractional Brownian motion  with Hurst parameter $H$.
Then, for any $s,\,t\in[0,1]$, we have
\begin{equation}\label{eq-2-17}
|h(t)|\leq \left(\mathbb{E}\left(|B_t^H|^2\right)\right)^{1/2}\Vert \dot{h}\Vert_{L^2}=t^H\Vert h\Vert_{\mathcal{H}_H}\leq \Vert h\Vert_{\mathcal{H}_H},
\end{equation}
and
\[
|h(t)-h(s)|\leq \left(\mathbb{E}\left(|B_t^H-B_s^H|^2\right)\right)^{1/2}\Vert \dot{h}\Vert_{L^2}=|t-s|^H\Vert \dot{h}\Vert_{L^2}=|t-s|^H\Vert h\Vert_{\mathcal{H}_H},
\]
which completes the proof.
\end{proof}\newline

Recall that $\Omega=C_0([0,1]:\RR^d)$ is a Banach space equipped with the sup-norm $\Vert \cdot \Vert_\infty$.  Let $\Omega^\ast$ be its topological dual. 

We now introduce the abstract Wiener space associated with a fractional Brownian motion.
The following result  is taken from \cite[Theorem 3.3]{DU}.
\begin{lemma}If we  identify  $L^2([0,1]:\RR^d)$ and its dual, we have the following diagram
\[
\Omega^\ast\overset{i_H^\ast}{ \xrightarrow{\hspace{1cm}}} \mathcal{H}_H^\ast\overset{K_H^\ast}{ \xrightarrow{\hspace{1cm}}}L^2([0,1]:\RR^d)\overset{K_H}{ \xrightarrow{\hspace{1cm}}}\mathcal{H}_H\overset{i_H}{ \xrightarrow{\hspace{1cm}}}\Omega
\]
Where $K_H$ is defined by \eqref{K-H-mapping}, $i_H$ is the injection from $\mathcal{H}_H$ into $\Omega$, and $K_H^\ast$ and $i_H^\ast$ are the respective adjoints. 
\begin{itemize}
\item[(a)] The injection $i_H$ embeds $\mathcal{H}_H$ densely into $\Omega$, and $\mathcal{H}_H$ is the Cameron-Martin space of the abstract Wiener space $(i_H,\mathcal{H}_H, \Omega)$ in the sense of Gross \cite{G}.
\item[(b)] The restriction of $K_H^\ast$ to $\Omega^\ast$ can be represented by 
\[
(K_H^\ast\eta)(s)=\int_0^1K_H(t,s)\eta(dt)=\left(\int_0^1K_H(t,s)\eta_1(dt),\dots,\int_0^1K_H(t,s)\eta_d(dt)\right), 
\]
for any $\eta=(\eta_1,\dots,\eta_d)\in\Omega^\ast$.
\end{itemize}
\end{lemma}

\bigskip

Next, we introduce the Skorohod integral $\delta(h)$ for  $h\in\mathcal{H}_H$. This is well known (cf. \cite{DU}, \cite{Zhang} and references therein)  but we give a self contained presentation for reader's convenience.
 The injection $i_H^\ast$ embeds $\Omega^\ast$ densely into $\mathcal{H}_H^\ast$, since $\Omega^\ast$ separates $\mathcal{H}_H$. Define $R_H=K_H\circ K_H^\ast$, then $R_H$ embeds $\Omega^\ast$ densely into $\mathcal{H}_H$, namely, for any $h\in\mathcal{H}_H$, there exists a sequence of $\{\eta_n\}_{n=1}^\infty\subset\Omega^\ast$ such that $\lim\limits_{n\to\infty}\Vert R_H \eta_n-h\Vert_{\mathcal{H}_H}=0$. 
Next note that
\begin{align*}
\int |\eta_n(\omega)-\eta_k(\omega)|^2\mathbb{P}(d\omega)&=\mathbb{E}(|\eta_n(\cdot)-\eta_k(\cdot)|^2)\\
&=\mathbb{E}\left(\sum_{1\leq i,j\leq d}\int_0^1\int_0^1\omega^i(t)\omega^j(s)(\eta_n^i-\eta_k^i)(dt)(\eta_n^j-\eta_k^j)(ds)\right)\\
&=\sum_{1\leq i,j\leq d}\int_0^1\int_0^1R_H^{i,j}(t,s)(\eta_n^i-\eta_k^i)(dt)(\eta_n^j-\eta_k^j)(ds)\\
&=\sum_{1\leq i,j\leq d}\int_0^1\int_0^1\int_0^1\delta_{i,j}K_H(t,r)K_H(s,r)dr\,(\eta_n^i-\eta_k^i)(dt)(\eta_n^j-\eta_k^j)(ds)\\
&=\int_0^1|K_H^\ast(\eta_n-\eta_k)(r)|^2dr=\Vert R_H(\eta_n-\eta_k)\Vert_{\mathcal{H}_H}^2\to 0,
\end{align*}
as $n,k\to\infty$. Therefore, there exists $\delta(h)\in L^2(\Omega, \mathcal{F},\mathbb{P})$ such that $\lim\limits_{n\to\infty}\mathbb{E}(|\eta_n(\cdot)-\delta(h)|^2)=0$. The limit $\delta(h)$ is called the Skorohod integral of $h$. Moreover, the following isometry property holds
\begin{equation}
	\label{eq:isom}
\mathbb{E}(\delta(h)\delta(g))=\langle h,g\rangle_{\mathcal{H}_H},  \mbox{ for any } h,\,g\in\mathcal{H}_H.
\end{equation}

Recall the filtered probability space $(\Omega,\mathcal{F},\mathbb{P},\{\mathcal{F}^H_t\})$
introduced below Definition \ref{def:fbm} and recall that 
$\{B^H_t\}$ is the canonical coordinate process
on $(\Omega,\mathcal{F})$.

Define the family $\{\pi_t^H,\, t\in[0,1]\}$ of orthogonal projections in $\mathcal{H}_H$ by
\begin{equation}\label{eq:projfam}
\pi_t^Hh=\pi_t^H(K_H\dot{h})=K_H(\dot{h}\mathbf{1}_{[0,t]}), \; h\in\mathcal{H}_H.
\end{equation}

Define for $i= 1, \ldots d$, $e_i:[0,1]\to \RR^d$ as 
$$(e_i(t))_j = \begin{cases}
0 & \mbox{ if } j \neq i\\
1 & \mbox{ if } j=i
\end{cases},\; t \in [0,1].
$$
Then  $e_i \in L^2([0,1]:\RR^d)$ for $i= 1, \ldots d$. Define the process $B=(B^1,\dots,B^d)$ by
 \begin{equation}\label{bm}
 B^i_t=\delta(\pi^H_tK_He_i), \; t\in[0,1], \; i=1,\dots,d.
 \end{equation}
 
If $h \in \clh_H$ then, it can be shown that, $g$ defined as 
$$g(t) = c_H^{-1}t^{H-\frac12}(D_{0+}^{H-\frac12}(\psi^{-1} h'))(t),\;\; t \in [0,1]$$
is in $L^2([0,1]:\RR^d)$ and we have $h = K_H g$. Henceforth, for $h \in \clh_H$, we will take the function $g$ defined as above the definition of $\dot h$ in the representation $h = K_H \dot h$.

 The following result is a consequence of
  \cite[Proposition 4.4, Theorems 4.3 and 4.8]{DU}.

\begin{lemma}\label{filtration} 
\begin{itemize}
\item[(a)] For any $t\in[0,1]$,  $\mathcal{F}_t^H=\sigma\{\delta(\pi^H_t h),\,h\in\mathcal{H}_H\}\vee \mathcal{N}$. 
\item[(b)] A $\mathcal{H}_H$-valued stochastic process $u=K_H\dot{u}$ is $\{\mathcal{F}^H_t,\,t\in[0,1]\}$-adapted if 
and only if the process $\dot{u}$ is $\{\mathcal{F}^H_t,\,t\in[0,1]\}$-progressively measurable.  
\item[(c)] The process $B=\{B_t=(B^1_t,\dots,B^d_t),\,t\in[0,1]\}$ is a $\mathbb{P}$-standard Brownian motion and
$\sigma\{B_s, s \le t\}\vee \mathcal{N}$
 equals  $\mathcal{F}^H_t$ for $t\in[0,1]$.
\end{itemize}
\end{lemma}

\begin{proof}[Proof of Proposition \ref{representation}] 
The collection  $\{\pi_t^H,\, t\in[0,1]\}$ introduced in \eqref{eq:projfam} defines 
 a continuous and strictly monotonic resolution 
 of the identity in $\mathcal{H}_H$ in the sense of \cite[Section 2]{Zhang} and furthermore the operator
 $\delta(\cdot)$ introduced in the latter paper coincides with the Skorohod integral defined above \eqref{eq:isom}.
 Also, the classes $\cla$ and $\cla_b$ in Section \ref{sec:varrep} are the same as the classes $\clh^a$ and $\clh^a_b$ in \cite{Zhang} when specialized to the setting considered here. The result is now immediate from
 \cite[Theorem 3.2]{Zhang}.
\end{proof}

\subsection{Proof of Theorem \ref{general-LP}}
The proof follows along the lines of   \cite[Theorem 4.4]{BD} however we provide details for reader's convenience. It will be convenient to work with the following equivalent formulation of a LDP.
\begin{definition}
 Let $I$ be a rate function on some Polish space $\mathcal{E}$. A collection $\{X^\varepsilon:\, \varepsilon\in[0,1]\}$ of $\cle$-valued random variables is said to satisfy the Laplace principle upper bound (lower bound, respectively) on $\mathcal{E}$ with rate function $I$ if for all bounded continuous functions $h:\,\mathcal{E}\rightarrow \mathbb{R}$, 
\begin{equation}\label{upper-bound}
\limsup_{\varepsilon\to0} -\varepsilon\log{\mathbb{E}\left(\exp\left\{-\frac{h(X^\varepsilon)}{\varepsilon}\right\}\right)}\leq \inf_{x\in\mathcal{E}}\{h(x)+I(x)\},
\end{equation}
(respectively,
\begin{equation}\label{lower-bound}
\liminf_{\varepsilon\to0} -\varepsilon\log{\mathbb{E}\left(\exp\left\{-\frac{h(X^\varepsilon)}{\varepsilon}\right\}\right)}\geq\inf_{x\in\mathcal{E}}\{h(x)+I(x)\}.)
\end{equation}
The Laplace principle (LP) is said to hold for $\{X^\varepsilon:\,\varepsilon\in[0,1]\}$ with rate function $I$ if both the Laplace upper bound and lower bound are satisfied for all bounded continuous functions $h:\,\mathcal{E}\rightarrow \mathbb{R}$.	

\end{definition}
 It is well known that a collection $\{X^\varepsilon:\,\varepsilon\in[0,1]\}$ of $\cle$-valued random variables satisfies a LDP with rate function
 $I$ if and only if it satisfies a LP with rate function $I$ (cf. \cite[Theorems 1.5 and 1.8]{buddupbook}). \\
 
 We now proceed to the proof of Theorem \ref{general-LP}.

\begin{proof}[Proof of Theorem \ref{general-LP}]
It suffices to prove \eqref{upper-bound} and \eqref{lower-bound}, with $I$ as defined in \eqref{eq-rate}, for all real-valued, bounded and continuous functions $h$ on $\mathcal{E}$, and to prove that $I$ is a rate function.\\

{\em Proof of the upper bound \eqref{upper-bound}:} Without loss we assume that $\inf_{x\in\mathcal{E}}\{h(x)+I(x)\}<\infty$.
Let $\delta>0$ be arbitrary. Then there exists $x_0\in\mathcal{E}$ such that
\begin{equation}\label{s-may-1}
h(x_0)+I(x_0)\leq \inf_{x\in\mathcal{E}}\{h(x)+I(x)\}+\frac{\delta}{2}<\infty.
\end{equation}
From the definition of $I$ there exists $\tilde{v}\in \mathcal{H}_H$ such that 
\[
\frac{1}{2}\Vert \tilde{v}\Vert_{\mathcal{H}_H}^2\leq I(x_0)+\frac{\delta}{2},
\mbox{ and }
x_0=\mathcal{G}^0(\tilde{v}).
\]
Applying Proposition \ref{representation} to the function $h\circ \clg^\varepsilon$, one has
\begin{align}\label{equ-2-27}
 -\varepsilon\log \mathbb{E}\left(\exp\left\{-\frac{h(X^\varepsilon)}{\varepsilon}\right\}\right)=& -\varepsilon\log \mathbb{E}\left(\exp\left\{-\frac{h\circ \clg^\varepsilon(\sqrt{\varepsilon}\omega)}{\varepsilon}\right\}\right)\notag\\
 =&\ \inf_{v\in\mathcal{A}_b}\mathbb{E}\left(h\circ \mathcal{G}^\varepsilon\left(\sqrt{\varepsilon} \omega+v\right)+\frac{1}{2}\Vert v\Vert_{\mathcal{H}_H}^2\right).
\end{align}
Thus, we have
\begin{eqnarray}\label{equ-2-25}
\limsup_{\varepsilon\to0} -\varepsilon\log \mathbb{E}\left(\exp\left\{-\frac{h(X^\varepsilon)}{\varepsilon}\right\}\right)
&=&\limsup_{\varepsilon\to0}\inf_{v\in\mathcal{A}_b}\mathbb{E}\left(h\circ \mathcal{G}^\varepsilon\left(\sqrt{\varepsilon}\omega+v\right)+\frac{1}{2}\Vert v\Vert_{\mathcal{H}_H}^2\right)\nonumber\\
&\leq&\limsup_{\varepsilon\to0}\mathbb{E}\left(h\circ \mathcal{G}^\varepsilon\left(\sqrt{\varepsilon}\omega+\tilde{v}\right)+\frac{1}{2}\Vert \tilde{v}\Vert_{\mathcal{H}_H}^2\right)\nonumber\\
&\leq&\limsup_{\varepsilon\to0}\mathbb{E}\left(h\circ \mathcal{G}^\varepsilon\left(\sqrt{\varepsilon}\omega+\tilde{v}\right)\right)+I(x_0)+\frac{\delta}{2}.
\end{eqnarray}
Since $h$ is bounded and continuous, from (i) in Assumption \ref{assum-G}, as $\varepsilon$ go to $0$ the last term in the above inequality equals
\begin{equation}\label{equ-2-26}
h\circ \mathcal{G}^0\left(\tilde{v}\right)+I(x_0)+\frac{\delta}{2}=h(x_0)+I(x_0)+\frac{\delta}{2}.
\end{equation}
Combining \eqref{s-may-1}, \eqref{equ-2-25} and \eqref{equ-2-26}, we obtain
\[
\limsup_{\varepsilon\to0} -\varepsilon\log \mathbb{E}\left(\exp\left\{-\frac{h(X^\varepsilon)}{\varepsilon}\right\}\right)\leq \inf_{x\in\mathcal{E}}\{h(x)+I(x)\}+\delta.
\]
Since $\delta$ is arbitrary, the upper bound holds.\newline

{\em Proof of the lower bound \eqref{lower-bound}:} Fix $\delta>0$. Then for each $\varepsilon$ there exist $v^\varepsilon\in\mathcal{A}_b$ such that
\begin{equation}\label{equ-2-28}
\inf_{v\in\mathcal{A}_b}\mathbb{E}\left(h\circ \mathcal{G}^\varepsilon\left(\sqrt{\varepsilon} \omega+v\right)+\frac{1}{2}\Vert v\Vert_{\mathcal{H}_H}^2\right)\geq \mathbb{E}\left(h\circ \mathcal{G}^\varepsilon\left(\sqrt{\varepsilon} \omega+v^\varepsilon\right)+\frac{1}{2}\Vert v^\varepsilon\Vert_{\mathcal{H}_H}^2\right)-\delta.
\end{equation}

Next note  that from \eqref{equ-2-27}  and \eqref{equ-2-28},
 for all $\varepsilon$,  $\mathbb{E}\left(\frac{1}{2}\Vert v^\varepsilon\Vert_{\mathcal{H}_H}^2\right)\leq 2M+\delta$, where $M=\Vert h\Vert_\infty$. Now define stopping times $\tau^\varepsilon_N=\inf\{t\in[0,1]: \frac{1}{2}\int_0^t |\dot{v}^\varepsilon(s)|^2ds\geq N\}\wedge 1$, where $v^\varepsilon=K_H\dot{v}^\varepsilon$. Denote $\dot{v}^{\varepsilon,N}=\dot{v}^\varepsilon1_{[0,\tau^\varepsilon_N]}(s)$ and $v^{\varepsilon,N}=K_H\dot{v}^{\varepsilon,N}$. Then the processes $v^{\varepsilon,N}$ are in $\mathcal{A}_b$ with $\frac{1}{2}\Vert v^{\varepsilon,N}\Vert_{\mathcal{H}_H}^2\leq N$ a.s, and moreover
\[
\mathbb{P}(v^{\varepsilon}\neq v^{\varepsilon,N})\leq\mathbb{P}\left(\frac{1}{2}\int_0^1|\dot{v}^\varepsilon(s)|^2ds\geq N\right)=\mathbb{P}\left(\frac{1}{2}\Vert v^\varepsilon\Vert_{\mathcal{H}_H}^2\geq N\right)\leq \frac{2M+\delta}{N}.
\]
Choosing $N$ large enough so that $\frac{2M(2M+\delta)}{N} \le \delta$, we see that \eqref{equ-2-28}
holds with $v^{\varepsilon}$ replaced with $v^{\varepsilon,N}$
and $\delta$ with $2\delta$. 
Henceforth, we will suppress $N$ and denote $v^{\varepsilon,N}$ as $v^{\varepsilon}$.
Note that, by definition,
\begin{equation}\label{v-bound}
\sup_{\veps>0}\frac{1}{2}\Vert v^\varepsilon\Vert_{\mathcal{H}_H}^2\leq N \mbox{ a.s. }
\end{equation}

Since $\delta>0$ is arbitrary, in order to prove the lower bound \eqref{lower-bound}, we now only need to show that
\begin{align}\label{equ-2-28-1}
\liminf\limits_{\varepsilon\to 0} \mathbb{E}\left(h\circ \mathcal{G}^\varepsilon\left(\sqrt{\varepsilon} \omega+v^\varepsilon\right)+\frac{1}{2}\Vert v^\varepsilon\Vert_{\mathcal{H}_H}^2\right)\geq \inf_{x\in\mathcal{E}}\{h(x)+I(x)\}.
\end{align}

Choose a subsequence (still relabelled by $\varepsilon$) along which $v^{\varepsilon}$ converges in distribution to $v$ as $S_N$-valued random elements. Since $h$ is a bounded and continuous function and the function on $\mathcal{H}_H$ defined by $v\mapsto \frac{1}{2}\Vert v\Vert^2_{\mathcal{H}_H}$ is lower semi-continuous with respect to the weak topology, using (i) in Assumption \ref{assum-G} and Fatou's lemma, we  obtain
\begin{eqnarray*}
\liminf_{\varepsilon\to0}\mathbb{E}\left(h\circ \mathcal{G}^\varepsilon\left(\sqrt{\varepsilon}\omega+v^{\varepsilon}\right)+\frac{1}{2}\Vert v^{\varepsilon}\Vert_{\mathcal{H}_H}^2\right)
&\geq& \mathbb{E}\left(h\circ \mathcal{G}^0(v)+\frac{1}{2}\Vert v\Vert_{\mathcal{H}_H}^2\right)\nonumber\\
&\geq &\inf_{\{(x,v): x=\mathcal{G}^0(v)\}}\left\{h(x)+\frac{1}{2}\Vert v\Vert_{\mathcal{H}_H}^2\right\}\nonumber\\
&\geq&\inf_{x\in\mathcal{E}}\{h(x)+I(x)\}.
\end{eqnarray*}
This completes the proof of the lower bound.\newline

{\em Compactness of Level Sets:} In order to prove that $I$ is a rate function, we need to show that, for any $0<M<\infty$, the level set $\{x:\, I(x)\leq M\}$ is compact. In order to prove this,  we will show that
\[
\{x:\, I(x)\leq M\}=\bigcap_{n=1}^\infty \Gamma_{M+\frac{1}{n}}.
\]
The compactness of the level set $\{x:\, I(x)\leq M\}$ will then follow, since from (ii) in Assumption \ref{assum-G} the set $\Gamma_{M+\frac{1}{n}}$ is compact for each $n$.

Let $x\in \mathcal{E}$ with $I(x)\leq M$. Then, for each $n$, from the definition of $I(x)$ there exists $v^n\in\mathcal{H}_H$ such that $\frac{1}{2}\Vert v^n\Vert_{\mathcal{H}_H}^2\leq M+\frac{1}{n}$ and $x=\mathcal{G}^0(v^n)$. This shows $x\in \bigcap_{n=1}^\infty \Gamma_{M+\frac{1}{n}}$. Conversely, suppose  $x\in \bigcap_{n=1}^\infty \Gamma_{M+\frac{1}{n}}$. Then, for each $n$, there exists $v^n\in S_{M+\frac{1}{n}}$ such that $x=\mathcal{G}^0(v^n)$. Thus, we have $I(x)\leq \frac{1}{2}\Vert v^n\Vert_{\mathcal{H}_H}^2\leq M+\frac{1}{n}$, for all $n$. By letting $n$ go to infinity, we conclude that $I(x)\leq M$. This completes the proof that $I$ is a rate function.
\end{proof}
\subsection{Proof of Proposition \ref{existence-uniqueness}(b)}
\label{sec:pfpartb}
\begin{proof}
Without loss of generality we assume that $d=m=1$. Now, fix $0\leq s<t\leq 1$. In the proof, we will use $C$ to denote a generic positive constant which may depend on $\alpha$ and all the constants appearing in Assumption \ref{assum-b-sigma} but is independent of $s$ and $t$. This generic constant may vary from line to line.

From  \eqref{deterministic}, we have
\begin{equation}\label{eqn-m-1}
|x_t-x_s|\leq \int_s^t|b(r,x_r)|dr+\left|\int_s^t\sigma(r,x_r)dg_r\right|.
\end{equation}
Using the Lipschitz property and linear growth of $b$ we see that
\begin{eqnarray}\label{eqn-m-2}
\int_s^t|b(r,x_r)|dr&\leq &L\int_s^t(1+|x_s|+\Vert x\Vert_{s,t,1-\alpha}(r-s)^{1-\alpha})dr\notag\\
&\leq& C(1+|x_s|)(t-s)+C\Vert x\Vert_{s,t,1-\alpha}(t-s)^{2-\alpha}.
\end{eqnarray}
Recalling from Remark \ref{rem:exisuniq} that
$s\mapsto \sigma(s, x(s)) \in C^{\delta}([0,1]:\RR^m)$ for some $\delta> 1-H$
and using the fractional integration by parts formula given in Proposition \ref{p1}, we obtain
\begin{eqnarray*}
\left|\int_s^t\sigma(r,x_r)dg_r\right|&\leq&\int_s^t\left|D_{s+}^\alpha\sigma(r,x_r)D_{t-}^{1-\alpha}g_{t-}(r)\right|dr.
\end{eqnarray*}
From \eqref{1.1} and the conditions on $\sigma$ in Assumption \ref{assum-b-sigma}(i) we see   that, for $s\le r\le t$
\begin{eqnarray*}
\left|D_{s+}^\alpha\sigma(r,x_r)\right|&\leq& C\left(\frac{K(1+|x_s|)+M|x_r-x_s|}{(r-s)^\alpha}+\alpha\int_s^r\frac{M(r-u)^\lambda+M|x_r-x_u|}{(r-u)^{\alpha+1}}du\right)\\
&\leq&C(1+|x_s|)(r-s)^{-\alpha}+C(r-s)^{\lambda-\alpha}+C\Vert x\Vert_{s,t,1-\alpha}(r-s)^{1-2\alpha}.
\end{eqnarray*}
It is easy to verify that
$
\left|D_{t-}^{1-\alpha}g_{t-}(r)\right| \leq C\Vert g\Vert_{1-\alpha+\delta}(t-r)^\delta.
$
Thus,
\begin{eqnarray}\label{eqn-m-3}
\left|\int_s^t\sigma(r,x_r)dg_r\right|&\leq &C\Vert g\Vert_{1-\alpha+\delta}(1+|x_s|)(t-s)^{1-\alpha+\delta} +C\Vert g\Vert_{1-\alpha+\delta}(t-s)^{1+\lambda-\alpha+\delta}\notag\\
&&+C\Vert g\Vert_{1-\alpha+\delta}\Vert x\Vert_{s,t,1-\alpha}(t-s)^{2-2\alpha+\delta}\notag\\
&\leq &C\Vert g\Vert_{1-\alpha+\delta}(1+|x_s|)(t-s)^{1-\alpha+\delta}+C\Vert g\Vert_{1-\alpha+\delta}\Vert x\Vert_{s,t,1-\alpha}(t-s)^{2-2\alpha+\delta}.
\end{eqnarray}
Hence,  from \eqref{eqn-m-1}-\eqref{eqn-m-3} we get
\begin{eqnarray}\label{eqn-m-4}
\Vert x\Vert_{s,t,1-\alpha}\leq C_0(1+\Vert g\Vert_{1-\alpha+\delta})(1+|x_s|)+C_0(1+\Vert g\Vert_{1-\alpha+\delta})\Vert x\Vert_{s,t,1-\alpha}(t-s)^{1-\alpha},
\end{eqnarray}
for some $C_0>0$ independent of $s,t$.
Choose 
\[
\Delta=\left(\frac{1}{2C_0(1+\Vert g\Vert_{1-\alpha+\delta})}\right)^{\frac{1}{1-\alpha}}.
\]
 Then, for all $s, t$ with $t-s\leq \Delta$, we have
 \begin{equation}\label{eq-m-6}
 \Vert x\Vert_{s,t,1-\alpha}\leq 2C_0(1+\Vert g\Vert_{1-\alpha+\delta})(1+|x_s|).
 \end{equation}
 Therefore, for $s,t$ as above,
 \[
 \Vert x\Vert_{s,t,\infty}\leq |x_s|+\Vert x\Vert_{s,t,1-\alpha}(t-s)^{1-\alpha}\leq |x_s|+2C_0(1+\Vert g\Vert_{1-\alpha+\delta})(1+|x_s|)(t-s)^{1-\alpha}.
 \]
In particular, for all $s, t$ with $t-s\leq \Delta$, we  have
 $\Vert x\Vert_{s,t,\infty}\leq 2|x_s|+1,
 $
and hence
\begin{equation}\label{eqn-m-5}
\sup\limits_{0\leq r\leq t}|x_r|\leq 2\sup\limits_{0\leq r\leq s}|x_r|+1.
\end{equation}
Now  divide the interval $[0,1]$ into $n=\left[\frac{1}{\Delta}\right]+1$ subintervals with equal lengths, and  denote $t_k=\frac{k}{n}$, $k=0, 1, \dots, n$. Applying \eqref{eqn-m-5} in each subinterval we obtain
\[
\sup\limits_{0\leq r\leq 1}|x_r|\leq 2^n|x_0|+2^{n-1}+2^{n-2}+\dots+1\leq 2^n(|x_0|+1).
\]
Thus, recalling the definition of $n, \Delta$,
\begin{equation}\label{eq-m-7}
\Vert x\Vert_\infty\leq 2^{1+\left(2C_0(1+\Vert g\Vert_{1-\alpha+\delta})\right)^{\frac{1}{1-\alpha}}}(|x_0|+1).
\end{equation}
Now,  fix $u, v$ in $[0,1]$ with $u<v$. If $v-u\leq \Delta$ then \eqref{eq-m-6} and \eqref{eq-m-7} yield
\begin{eqnarray}\label{eq-m-8}
\frac{|x_v-x_u|}{(v-u)^{1-\alpha}}&\leq& 2C_0(1+\Vert g\Vert_{1-\alpha+\delta})(1+\Vert x\Vert_\infty)\notag\\
&\leq& 2C_0(1+\Vert g\Vert_{1-\alpha+\delta})\left(1+2^{1+\left(2C_0(1+\Vert g\Vert_{1-\alpha+\delta})\right)^{\frac{1}{1-\alpha}}}(|x_0|+1)\right).
\end{eqnarray}
If $v-u>\Delta$, there are  $k_0 =k_0(u,v), k_1=k_1(u,v)\in \{1,\dots, n\}$ such that $k_0 \leq k_1$,
$t_{k_0}-u \le 1/n$, $v- t_{k_1} \le 1/n$, and
\[
u<t_k<v,\  \mbox{for all }\  k=k_0, k_0+1, \dots, k_1.
\]
Then, applying \eqref{eq-m-6} and \eqref{eq-m-7} we  obtain
\begin{eqnarray}\label{eq-m-9}
\frac{|x_v-x_u|}{(v-u)^{1-\alpha}}&\leq& \frac{|x_v-x_{t_{k_1}}|}{(v-t_{k_1})^{1-\alpha}}+ \frac{|x_{t_{k_0}}-x_u|}{(t_{k_0}-u)^{1-\alpha}}+\sum_{k=k_0+1}^{k_1}\frac{|x_{t_k}-x_{t_{k-1}}|}{(t_k-t_{k-1})^{1-\alpha}}\notag\\
&\leq& 2nC_0(1+\Vert g\Vert_{1-\alpha+\delta})(1+\Vert x\Vert_\infty)\notag\\
&\leq&2C_0(1+\Vert g\Vert_{1-\alpha+\delta})\left(\left(2C_0(1+\Vert g\Vert_{1-\alpha+\delta})\right)^{\frac{1}{1-\alpha}}+2\right)\notag\\
&&\times\left(1+2^{1+\left(2C_0(1+\Vert g\Vert_{1-\alpha+\delta})\right)^{\frac{1}{1-\alpha}}}(|x_0|+1)\right).
\end{eqnarray}
The estimates in \eqref{eq-m-7}-\eqref{eq-m-9} imply \eqref{equ-4-2} and \eqref{equ-m-4-2} completing the proof.
\end{proof}

\subsection{Proof of Theorem \ref{LPmult}}
Throughout this subsection we assume that Assumption \ref{assum-b-sigma} holds.
The following lemma is immediate from the  strong existence and pathwise uniqueness result in Proposition \ref{prop-sde}. 
 \begin{lemma}\label{lem-4-3} 
For each $\varepsilon\in(0,1)$, there exists a measurable function
\[
\mathcal{G}^\varepsilon:\,C_0([0,1]:\RR^d)\to C([0,1]:\RR^m)
\]
such that, for any  probability space $(\Omega, \mathcal{F},\mathbb{P})$ with a fractional Brownian motion $B^H$ defined on it, 
\begin{equation}\label{equ-4-6-0}
X^\varepsilon=\mathcal{G}^\varepsilon(\sqrt{\varepsilon}B^H)
\end{equation}
is the unique pathwise solution of \eqref{small-noise SDE}.
Moreover,  for any $\alpha\in(1-H,\min\{\frac{1}{2},\lambda,\frac{\gamma}{1+\gamma}\})$,  the following estimate holds:
\begin{equation}\label{equ-4-6}
\sup\limits_{\varepsilon\in(0,1)}\mathbb{E}\Vert X^\varepsilon\Vert_{1-\alpha}^p<\infty, \ \mbox{for all } \ p\geq 1.
\end{equation}
 \end{lemma}

\begin{remark} 
The estimate \eqref{equ-4-6} implies the tightness of $\{X^{\varepsilon}\}$ in $C([0,1]:\mathbb{R}^m)$. 
  Moreover, by a similar argument as in the proof of Lemma \ref{lem-4-5}, it can be seen that the sequence of processes $\{X^\varepsilon\}$ converges in distribution to the solution $\phi$ to the following ordinary differential equation
 \begin{equation}\label{ODE}
\phi_t=x_0+\int_0^tb(s,\phi_s)ds.
\end{equation}
The existence and uniqueness of the solution to \eqref{ODE} follows from the Lipschitz condition on $b$
in Assumption \ref{assum-b-sigma}.
\end{remark}
The following  Girsanov theorem is a multi-dimensional version of \cite[Theorem 4.9]{DU}. 
Recall the operator $K_H$ from \eqref{K-H-mapping} and the canonical space $(\Omega,\mathcal{F},\{\mathcal{F}^H_t\}, \mathbb{P})$ along with  the canonical process $B^H$ on this space from Section \ref{sec:4.1}.
\begin{theorem}\label{Girsanov}
Let    $u=K_H\dot{u} \in \mathcal{A}$ be such that
\[
\mathbb{E}\left(\exp\left\{-\int_0^1\dot{u}(s)\cdot dB_s-\frac{1}{2}\int_0^1|\dot{u}(s)|^2ds\right\}\right)=1,
\]
where $B$ is the $d$ -dimensional standard Brownian motion defined by \eqref{bm}
on the probability space $(\Omega,\mathcal{F},\{\mathcal{F}^H_t\}, \mathbb{P})$ and 
\[
\int_0^1\dot{u}_s\cdot dB_s=\sum_{i=1}^d\int_0^1\dot{u}^i_sdB^i_s
\]
is  the usual It\^{o} integral.
 Let $\tilde{\mathbb{P}}$ be the probability measure on $(\Omega,\mathcal{F})$ defined by
\[
\frac{d\tilde{\mathbb{P}}}{d\mathbb{P}}=\exp\left\{-\int_0^1\dot{u}(s)\cdot dB_s-\frac{1}{2}\int_0^1|\dot{u}(s)|^2ds\right\}.
\]
Then the law of the process 
\begin{equation}\label{tilde-fbm}
\tilde{B}^H=\left\{\tilde{B}^H_t=B^H_t+u_t=B^H_t+\int_0^tK_H(t,s)\dot{u}_sds:\, t\in[0,1]\right\}
\end{equation}
under the probability $\tilde{\mathbb{P}}$ is the same as the law of the process $B^H$ under the probability $\mathbb{P}$. Namely, the process $\tilde{B}^H$ is a fractional Brownian motion under $\tilde{\mathbb{P}}$.
\end{theorem}

Let $v\in\mathcal{A}_b$ and consider the  controlled version of the SDE \eqref{small-noise SDE}  given as
\begin{equation}\label{v-equation}
X_t^{\varepsilon,v}=x_0+\int_0^tb(s,X_s^{\varepsilon,v})ds+\int_0^t\sigma(s,X_s^{\varepsilon,v})dv_s+\sqrt{\varepsilon}\int_0^t\sigma(s,X_s^{\varepsilon,v})dB^H_s.
\end{equation}
The following result gives the wellposedness of the above equation.
\begin{lemma}\label{lem-4-4}
Let $\mathcal{G}^\varepsilon$ be as in Lemma \ref{lem-4-3} and let $v\in\mathcal{A}_b$. Define 
\[
X^{\varepsilon,v}=\mathcal{G}^\varepsilon(\sqrt{\varepsilon}B^H+v).
\]
Then $X^{\varepsilon,v}$ is the unique pathwise solution of \eqref{v-equation}. Moreover, for any fixed $\alpha$ with $1-H<\alpha<\min\{\frac{1}{2},\lambda,\frac{\gamma}{1+\gamma}\}$and for any $0<\delta<\alpha-(1-H)$,  we have, a.s.,
\begin{eqnarray}\label{equ-4-8}
\Vert X^{\varepsilon,v}\Vert_{1-\alpha}&\leq& \tilde C_1(1+|x_0|)(1+\Vert B^H\Vert_{1-\alpha+\delta}^\kappa+\Vert v\Vert_{\mathcal{H}_H}^\kappa)(1+\Vert B^H\Vert_{1-\alpha+\delta}+\Vert v\Vert_{\mathcal{H}_H})\notag\\
&&\times(1+\exp\{\tilde C_2\left(\Vert B^H\Vert_{1-\alpha+\delta}^\kappa+\Vert v\Vert_{\mathcal{H}_H}^\kappa\right)\}),
\end{eqnarray}
for all $\varepsilon$, where  $\kappa=\frac{1}{1-\alpha}$ and the constants $\tilde C_1$ and $\tilde C_2$ depend only on $\alpha,\,\delta$ and all the constants appearing in Assumption \ref{assum-b-sigma}.
\end{lemma}
\begin{proof}
Fix $v=K_H\dot{v}\in\mathcal{A}_b$. Note that
\[
\mathbb{E}\left(\exp\left\{-\frac{1}{\sqrt{\varepsilon}}\int_0^1\dot{v}_s\cdot dB_s-\frac{1}{2\varepsilon}\int_0^1|\dot{v}_s|^2ds\right\}\right)=1,
\]
where $B$ is the $d$ -dimensional standard Brownian motion defined by \eqref{bm}.
 Let $\tilde{\mathbb{P}}$ be the probability measure defined by
\[
\frac{d\tilde{\mathbb{P}}}{d\mathbb{P}}=\exp\left\{-\frac{1}{\sqrt{\varepsilon}}\int_0^1\dot{v}_s\cdot dB_s-\frac{1}{2\varepsilon}\int_0^1|\dot{v}_s|^2ds\right\}.
\]
Then, by Theorem \ref{Girsanov}, the process $\tilde{B}^H=B^H+\frac{1}{\sqrt{\varepsilon}}v$ is a fractional Brownian motion on $(\Omega,\mathcal{F},\tilde{\mathbb{P}},\{\mathcal{F}_t\})$. By Lemma \ref{lem-4-3}, the process $X^{\varepsilon,v}=\mathcal{G}^\varepsilon(\sqrt{\varepsilon}B^H+v)=\mathcal{G}^\varepsilon(\sqrt{\varepsilon}\tilde{B}^H)$ is the unique solution to \eqref{small-noise SDE} on $(\Omega,\mathcal{F},\tilde{\mathbb{P}},\{\mathcal{F}_t\})$. Note that the equation \eqref{small-noise SDE} with $\tilde{B}^H$ is precisely the same as the equation \eqref{v-equation}. Since the probability measures $\tilde{\mathbb{P}}$ and $\mathbb{P}$ are equivalent (i.e. mutually absolutely continuous), we see that $X^{\varepsilon,v}$ is a strong solution of \eqref{v-equation} on $(\Omega,\mathcal{F},\mathbb{P},\{\mathcal{F}_t\})$. Uniqueness of solutions of \eqref{v-equation} is argued similarly using the uniqueness result in Lemma \ref{lem-4-3} and  Girsanov's theorem (Theorem \ref{Girsanov}) once more.

Now fix $\alpha$ with $1-H<\alpha<\min\{\frac{1}{2},\lambda,\frac{\gamma}{1+\gamma}\}$ and $0<\delta<\alpha-(1-H)$.  From Lemma \ref{Holder}, it follows that
\[
\Vert v\Vert_{1-\alpha+\delta}\leq \Vert v\Vert_H\leq \Vert v\Vert_{\mathcal{H}_H}.
\]
Then, applying Proposition \ref{existence-uniqueness} (with $g_s = \sqrt{\veps}B^H_s + v_s$), we 
 obtain \eqref{equ-4-8} 
for some constants $\tilde C_1$ and $\tilde C_2$, depending only on  the constants $C_1, C_2$  in Proposition \ref{existence-uniqueness}.
\end{proof}
 
\bigskip
Next, for  $v\in\mathcal{H}_H$,  consider the following deterministic equation
\begin{equation}\label{equ-4-9}
X^{0,v}=x_0+\int_0^tb(s,X^{0,v}_s)ds+\int_0^t\sigma(s,X^{0,v}_s)dv_s.
\end{equation}
Due to Assumption \ref{assum-b-sigma} and the H\"{o}lder continuity of $v$ proved in Lemma \ref{Holder}, Proposition \ref{existence-uniqueness} implies the existence and uniqueness of the solution to \eqref{equ-4-9}. Furthermore, for any $\alpha\in(1-H,\min\{\frac{1}{2},\lambda,\frac{\gamma}{1+\gamma}\})$, from \eqref{equ-4-2} and  the H\"{o}lder continuity of $v$ it follows that 
\begin{equation}\label{equ-4-10}
\Vert X^{0,v}\Vert_{1-\alpha}\leq C_3(1+|x_0|)(1+\Vert v\Vert_{\mathcal{H}_H}^\kappa)(1+\Vert v\Vert_{\mathcal{H}_H})(1+\exp\left\{C_2\Vert v\Vert_{\mathcal{H}_H}^\kappa\right\} ),
\end{equation}
where  $\kappa=\frac{1}{1-\alpha}$ and  $C_2$, $C_3$ are the constants in Proposition \ref{existence-uniqueness}.

Now, define the map $\mathcal{G}^0:\,\mathcal{H}_H\to C([0,1]:\RR^m)$ by 
	\begin{equation}\label{G-0-4}
	\mathcal{G}^0(v)=X^{0,v}, \; v\in\mathcal{H}_H,
	\end{equation}
	 where $X^{0,v}$ is the unique solution to \eqref{equ-4-9}.
\begin{lemma}\label{lem-4-5}
 Let the map $\mathcal{G}^0$ be as in \eqref{G-0-4}.
  Then for each $N <\infty$, the restriction of the map $\mathcal{G}^0$ to $S_N$ is continuous.
\end{lemma}
\begin{proof} Let $v^n=K_H(\dot{v}^n)$ and $v=K_H\dot{v}$ be in $S_N$, and assume that $v^n$ converges to $v$ in $S_N$ (i.e. under the weak topology on $\mathcal{H}_H$).

Let $X^{0,v^n}$ be the solution to \eqref{equ-4-9} with $v$ replaced by $v^n$, that is, it satisfies the following equation
\begin{equation}\label{equ-4-12}
X^{0,v^n}_t=x_0+\int_0^t b(s,X^{0,v^n}_s)ds+\int_0^t\sigma(s,X^{0,v^n}_s)dv^n_s.
\end{equation}

Now, fix $\alpha\in(1-H,\min\{\frac{1}{2},\lambda,\frac{\gamma}{1+\gamma}\})$.   Note that $\sup_{n}\Vert v^n\Vert_{\mathcal{H}_H}\le \sqrt{2N} <\infty$, and hence, by \eqref{equ-4-10}, we see that $\sup_{n} \Vert X^{0,v^n}\Vert_{1-\alpha}<\infty$. Thus, the sequence $\{X^{0,v^n}\}$ is relatively compact in $C([0,1]:\mathbb{R}^m)$ and therefore for any subsequence of $\{X^{0,v^n}\}$, there is a  further subsequence (still relabeled by $n$) such that $X^{0,v^n}$ converges to $X$ in $C([0,1]:\mathbb{R}^m)$.

For  any $h = K_H \dot{h} \in \mathcal{H}_H$, \eqref{H-big} implies that $h$ is differentiable and 
\begin{equation}
h'(t) = c_Ht^{H-\frac{1}{2}}(I^{H-\frac{1}{2}}_{0+}(\psi^{-1}\dot{h}))(t)=\frac{c_Ht^{H-\frac{1}{2}}}{\Gamma(H-\frac12)}\int_0^t(t-s)^{H-\frac32}s^{\frac12-H}\dot{h}(s)ds,
\end{equation}
where $\psi$ is as in Lemma \ref{int-trans}.
Thus, if $f\in C([0,1]:\RR^d)$, the Riemann-Stieltjes integral $\int_0^1 f(s)dh(s)$ is well-defined and equals $\int_0^1f(s)h'(s)ds$.
  
From  the Lipschitz condition on $\sigma$ and changing the order of integration, we have 
  \begin{eqnarray}\label{equ-4-13}
 && \left|\int_0^t\sigma(s,X^{0,v^n}_s)dv^n_s-\int_0^t\sigma(s,X_s)dv_s\right|\nonumber\\
  &\leq& \left|\int_0^t\sigma(s,X^{0,v^n}_s)dv^n_s-\int_0^t\sigma(s,X_s)dv^n_s\right|+ \left|\int_0^t\sigma(s,X_s)dv^n_s-\int_0^t\sigma(s,X_s)dv_s\right|\nonumber\\
  &=&\frac{c_H}{\Gamma(H-\frac12)}\left|\int_0^t(\sigma(s,X^{0,v^n}_s)-\sigma(s,X_s))s^{H-\frac{1}{2}}\left(\int_0^s(s-u)^{H-\frac{3}{2}}u^{\frac{1}{2}-H}\dot{v}^n_udu\right)ds\right|\nonumber\\
  &&+\frac{c_H}{\Gamma(H-\frac12)}\left|\int_0^t\sigma(s,X_s)s^{H-\frac{1}{2}}\left(\int_0^s(s-u)^{H-\frac{3}{2}}u^{\frac{1}{2}-H}(\dot{v}^n_u-\dot{v}_u)du\right)ds\right|\nonumber\\
  &\leq &\frac{Mc_H}{\Gamma(H-\frac12)}\Vert X^{0,v^n}-X\Vert_\infty\int_0^tu^{\frac{1}{2}-H}|\dot{v}^n_u|\left(\int_u^t s^{H-\frac{1}{2}}(s-u)^{H-\frac{3}{2}}ds\right)du\nonumber\\
  &&+\frac{c_H}{\Gamma(H-\frac12)}\left|\int_0^tu^{\frac{1}{2}-H}\left(\int_u^t s^{H-\frac{1}{2}}(s-u)^{H-\frac{3}{2}}\sigma(s,X_s)ds\right)(\dot{v}^n_u-\dot v_u)du\right|.
  \end{eqnarray}
Observe that the first term on the right-hand side of \eqref{equ-4-13} is no more than
\begin{align}
&C\Vert X^{0,v^n}-X\Vert_{\infty}\int_0^1u^{\frac{1}{2}-H}|\dot{v}^n_u|\left(\int_u^1 (s-u)^{H-\frac{3}{2}}ds\right)du\nonumber\\
&\quad \leq C\Vert X^{0,v^n}-X\Vert_{\infty}\int_0^1u^{\frac{1}{2}-H}|\dot{v}^n_u|du\nonumber\\
&\quad \leq C\Vert X^{0,v^n}-X\Vert_{\infty}\Vert\dot{v}^n\Vert_{L^2}\notag\\
&\quad = C\Vert X^{0,v^n}-X\Vert_\infty\Vert v^n\Vert_{\mathcal{H}_H}\leq C\sqrt{2N}\Vert X^{0,v^n}-X\Vert_\infty,
\label{eq:nbd}
\end{align}
where $C$ is a  constant depending on $M$ and $H$ which may vary from line to line. 
From  the uniform convergence of $X^{0,v^n}$ to $X$ the right side of \eqref{eq:nbd}  converges to $0$
 and thus we obtain 
\begin{equation}\label{equ-4-15}
\lim\limits_{n\to\infty} \Vert X^{0,v^n}-X\Vert_\infty\int_0^tu^{\frac{1}{2}-H}|\dot{v}^n_u|\left(\int_u^t s^{H-\frac{1}{2}}(s-u)^{H-\frac{3}{2}}ds\right)du=0.
\end{equation}
For the second term on the right-hand side of \eqref{equ-4-13}, we first study the function (in $u$) $f^t(u)=1_{[0,t]}(u)\int_u^t s^{H-\frac{1}{2}}(s-u)^{H-\frac{3}{2}}\sigma(s,X_s)ds$. From \eqref{linear-sig}, we get
\begin{equation}\label{equ-4-16}
|f^t(u)|\leq K(1+\Vert X\Vert_\infty)1_{[0,t]}(u)\int_u^ts^{H-\frac{1}{2}}(s-u)^{H-\frac{3}{2}}ds\leq  C(1+\Vert X\Vert_\infty),
\end{equation}
where $C$ is a constant depending on $K$ and $H$. Thus it follows that the function  $u \mapsto u^{\frac{1}{2}-H}f^t(u)$ is in $L^2([0,1]:\RR^{m\times d})$. Thus, from the convergence of $v^n$ to $v$ (in the weak topology), we have, for every
$t\in [0,1]$, 
\begin{equation}\label{equ-4-17}
\lim\limits_{n\to\infty}\left|\int_0^tu^{\frac{1}{2}-H}\left(\int_u^t s^{H-\frac{1}{2}}(s-u)^{H-\frac{3}{2}}\sigma(s,X_s)ds\right)(\dot{v}^n_u-\dot v_u)du\right|=0.
\end{equation}
From \eqref{equ-4-13}, \eqref{equ-4-15}, \eqref{equ-4-17}, we see that, for each $t\in[0,1]$,
\begin{equation}\label{equ-4-18}
\lim\limits_{n\to\infty}\int_0^t\sigma(s,X^{0,v^n}_s)dv^n_s=\int_0^t\sigma(s,X_s)dv_s.
\end{equation}

From \eqref{equ-4-12}, \eqref{equ-4-18}, the uniform convergence of $X^{0,v^n}$ to $X$  and the Lipschitz condition on $b$, we have by sending $n$ to infinity that $X$ satisfies  equation \eqref{equ-4-9}. Recalling that \eqref{equ-4-9} has a unique solution, we get $X=X^{0,v}$.
Thus, by a standard subsequential argument, we see that  the full sequence  $\mathcal{G}^0(v^n)=X^{0,v^n}$  converges to $\mathcal{G}^0(v)=X^{0,v}$ in $C([0,1]:\mathbb{R}^m)$. The proof is complete.
 \end{proof}\\

\begin{lemma}\label{lem:sufftone}
	Fix   $N \in (0, \infty)$. Let $\{v^\varepsilon\}\subset \mathcal{A}_b$ be a family of $S_N$-valued random variables   on $(\Omega,\mathcal{F},\mathbb{P})$  such that $v^\varepsilon$ converges in distribution  to $v$ in the weak topology on $S_N$. Then
	 $\mathcal{G}^\varepsilon(\sqrt{\varepsilon}\omega+v^\varepsilon)$ converges to $\mathcal{G}^0(v)$ in distribution.
\end{lemma}
\begin{proof}
Note from Lemma \ref{lem-4-4} that $\mathcal{G}^\varepsilon(\sqrt{\varepsilon}\omega+v^\varepsilon)=X^{\varepsilon,v^\varepsilon}$, where $X^{\varepsilon,v^\varepsilon}$ is the unique solution to the following SDE:
\begin{equation}\label{v-eq-varepsilon}
X_t^{\varepsilon,v^\varepsilon}=x_0+\int_0^tb(s,X_s^{\varepsilon,v^\varepsilon})ds+\int_0^t\sigma(s,X_s^{\varepsilon,v^\varepsilon})dv_s+\sqrt{\varepsilon}\int_0^t\sigma(s,X_s^{\varepsilon,v^\varepsilon})dB^H_s.
\end{equation}

For any fixed $\alpha$ with $1-H<\alpha<\min\{\frac{1}{2},\lambda,\frac{\gamma}{1+\gamma}\}$, from \eqref{equ-4-8} and Fernique's theorem (cf. \cite{MR0413238}) we have
\begin{equation}\label{equ-m-4-6}
\sup\limits_{\varepsilon\in(0,1)}\mathbb{E}\Vert X^{\varepsilon,v^\varepsilon}\Vert_{1-\alpha}^p<\infty, \ \mbox{for all } \ p\geq 1.
\end{equation}
 In particular $\{X^{\varepsilon,v^\varepsilon}\}$ is tight in $C([0,1]:\mathbb{R}^m)$. This tightness together with the compactness of $S_N$ under topology of the weak convergence yield that for any subsequence of $\{(X^{\varepsilon,v^\varepsilon}, v^\varepsilon)\}$ there is a  further subsequence (still relabeled by $\varepsilon$) such that $(X^{\varepsilon,v^\varepsilon}, v^\varepsilon)$ converges weakly to $(X, v)$ in $C([0,1]:\mathbb{R}^m)\times S_N$.
 
Define the mapping $F^{b,\sigma}: C([0,1]:\mathbb{R}^m)\times S_N \to C([0,1]:\mathbb{R}^m)$ as   
\[
F^{b,\sigma}_t(x,u)=x_0+\int_0^tb(s, x_s)ds+\int_0^t\sigma(s, x_s)du_s, \ t\in [0,1],\; (x,u)\in C([0,1]:\mathbb{R}^m)\times S_N. 
\]
It is easy to see that the right side defines a function in  $C([0,1]:\mathbb{R}^m)$ for any  $(x,u)\in C([0,1]:\mathbb{R}^m)\times S_N$. 
We will now show that  $F^{b,\sigma}$ is a continuous mapping from $C([0,1]:\mathbb{R}^m)\times S_N$ to $C([0,1]:\mathbb{R}^m)$. Let $(x^n,u^n)$ converge to $(x,u)$ in $ C([0,1]:\mathbb{R}^m)\times S_N$, and denote $u^n=K_H\dot u^n$ and $u=K_H\dot u$ with $u^n, \, u\in L^2([0,1]:\RR^d)$.  By the Lipschitz condition on $b$ and analogous arguments as in \eqref{equ-4-13}, \eqref{equ-4-15} and \eqref{equ-4-17}, we obtain
\begin{align}
&\sup_{0\leq t\leq 1}|F^{b,\sigma}_t(x^n,u^n)-F^{b,\sigma}_t(x,u)|\notag\\
&\quad\leq L\Vert x^n-x\Vert_\infty+\frac{Mc_H}{\Gamma(H-\frac12)}\Vert x^n-x\Vert_\infty\int_0^1r^{\frac{1}{2}-H}|\dot{u}^n_r|\left(\int_r^1 s^{H-\frac{1}{2}}(s-r)^{H-\frac{3}{2}}ds\right)dr\nonumber\\
  &\quad\quad+\frac{c_H}{\Gamma(H-\frac12)}\left|\int_0^1r^{\frac{1}{2}-H}\left(\int_r^1 s^{H-\frac{1}{2}}(s-r)^{H-\frac{3}{2}}\sigma(s,x_s)ds\right)(\dot{u}^n_r-\dot u_r)dr\right|\to 0,
\end{align}
as $n\to\infty$, which proves the continuity of $F^{b,\sigma}$.
 Thus, by the continuous mapping theorem, $\{ F^{b,\sigma}_\cdot(X^{\varepsilon, v^\varepsilon},v^\varepsilon)=x_0+\int_0^\cdot b(s, X^{\varepsilon, v^\varepsilon}_s)ds+\int_0^\cdot\sigma(s, X^{\varepsilon, v^\varepsilon}_s)dv^\varepsilon_s\}$ converges in distribution to $\{ F^{b,\sigma}_\cdot(X,v)=x_0+\int_0^\cdot b(s, X_s)ds+\int_0^\cdot\sigma(s, X_s)dv_s\}$ in $C([0,1]:\mathbb{R}^m)$.
 
Now for any $0<\delta<\alpha-(1-H)$, from the observation in \eqref{eqn-m-3} with $(x, g)$ replaced by $(X^{\varepsilon, v^\varepsilon},B^H)$, we obtain
\begin{eqnarray*}
\left\Vert \int_0^\cdot\sigma(s,X^{\varepsilon, v^\varepsilon}_s)dB^H_s\right\Vert_{1-
\alpha}&\leq& C(1+\Vert X^{\varepsilon, v^\varepsilon}\Vert_\infty+\Vert X^{\varepsilon, v^\varepsilon}\Vert_{1-\alpha})\Vert B^H\Vert_{1-\alpha+\delta}\\
&\leq& C(1+\Vert X^{\varepsilon, v^\varepsilon}\Vert_{1-\alpha})\Vert B^H\Vert_{1-\alpha+\delta}.\\
\end{eqnarray*}
Then, from \eqref{equ-m-4-6} we get  
\[
\sup_{\varepsilon\in(0,1)}\mathbb{E}\left\Vert \int_0^\cdot\sigma(s,X^{\varepsilon, v^\varepsilon}_s)dB^H_s\right\Vert_{1-
\alpha}^p<\infty, \ \mbox{for all }\ p\geq 1,
\]
which implies that $\{\sqrt{\varepsilon}\int_0^\cdot\sigma(s,X^{\varepsilon, v^\varepsilon}_s)dB^H_s\}$ converges in probability to $0$ in $C([0,1]:\mathbb{R}^m)$.

Combining the above observations we now have that
\[
x_0+\int_0^\cdot b(s,X_s^{\varepsilon,v^\varepsilon})ds+\int_0^\cdot\sigma(s,X_s^{\varepsilon,v^\varepsilon})dv_s+\sqrt{\varepsilon}\int_0^\cdot\sigma(s,X_s^{\varepsilon,v^\varepsilon})dB^H_s
\]
 converges in distribution in $C([0,1]:\mathbb{R}^m)$ to 
 \[
x_0+\int_0^\cdot b(s, X_s)ds+\int_0^\cdot\sigma(s, X_s)dv_s.
 \]
 
 Since $X^{\varepsilon,v^\varepsilon}$ converges  to $X$ in distribution in $C([0,1]:\mathbb{R}^m)$ and that the solution to \eqref{equ-4-9} is unique, we conclude that $X^{\varepsilon,v^\varepsilon}=\mathcal{G}^\varepsilon(\sqrt{\varepsilon}\omega+v^\varepsilon)$ converges to $X^{0,v}=\mathcal{G}^{0}(v)$ in distribution in $C([0,1]:\mathbb{R}^m)$.
\end{proof}
$\,$\\ 

\noindent \begin{proof}[Proof of Theorem \ref{LPmult}] 
Note that the rate function $I$ defined in \eqref{eq:rfmul} coincides with the one in \eqref{eq-rate} for $\mathcal{G}^0$ defined in \eqref{G-0-4}.
From Theorem \ref{general-LP}, to prove Theorem \ref{LPmult}, it is sufficient to verify the conditions in Assumption \ref{assum-G} for $ \mathcal{G}^\varepsilon$ defined in Lemma \ref{lem-4-3}  and $\mathcal{G}^0$ in \eqref{G-0-4}.
Lemma \ref{lem:sufftone} shows that the Assupmtion \ref{assum-G}(i) is satisfied. Also,
since $S_N$ is compact in the weak topology, Lemma \ref{lem-4-5} implies that the set 
\[
\Gamma_N=\{\mathcal{G}^0(v): v\in S_N\}
\]
is a compact subset of $\mathcal{E}$, which verifies  Assumption \ref{assum-G}(ii), . 
\end{proof}

\section*{Acknowledgements}
Research  of AB is supported in part by the National Science Foundation (DMS-1814894 and DMS-1853968).

 \def\cprime{$'$} \def\cprime{$'$}

\end{document}